\documentclass[11pt, oneside, reqno]{amsart}

\newtheorem{thm}{Theorem}[section]
\newtheorem{main-thm}{Main Theorem}
\newtheorem{lem}[thm]{Lemma}
\newtheorem{prop}[thm]{Proposition}

\theoremstyle{definition}
\newtheorem{defn}[thm]{Definition}

\newtheorem{prob}{Problem}
\newtheorem{notn}[thm]{Notation}

\theoremstyle{remark}

\numberwithin{equation}{section}

\usepackage{latexsym}
\usepackage{amsmath}
\usepackage{amsthm}
\usepackage{amssymb}
\usepackage{amsfonts}
\usepackage[all]{xy}
\usepackage{graphicx}
\usepackage{lscape}
\usepackage{verbatim}
\usepackage{wrapfig}
\usepackage{subfig}
\usepackage{setspace}
\usepackage{pinlabel}
\usepackage{enumerate, setspace, dsfont}
\usepackage{fancyhdr}
\usepackage{calc}
\usepackage{rotating}

\begin{document}

\title{The simple complexity of a Riemann surface}
\author{Aldo-Hilario Cruz-Cota and Teresita Ramirez-Rosas}
\address{Department of Mathematics, Grand Valley State University, Allendale, MI 49401-9401, USA}
\email{cruzal@gvsu.edu}
\keywords{Riemann surfaces, complexity of branched covers, Hurwitz problem}
\address{Department of Mathematics, Grand Valley State University, Allendale, MI 49401-9401, USA}
\email{ramirezt@gvsu.edu}

\date{\today}

\dedicatory{}

\begin{abstract}

\noindent Given a Riemann surface $M$, the \emph{complexity} of a branched cover of $M$ to the Riemann sphere $S^2$, of degree $d$ and with branching set of  cardinality $n \geq 3$, is defined as $d$ times the hyperbolic area of the complement of its branching set in $S^2$. A branched cover $p \colon M \to S^2$ of degree $d$ is \emph{simple} if the cardinality of the pre-image $p^{-1}(y)$ is at least $d-1$ for all $y \in S^2$. The  \emph{(simple) complexity} of $M$ is defined as the infimum of the complexities of all (simple) branched covers of $M$  to $S^2$. We prove that if $M$ is a closed, connected,  orientable Riemann surface of genus $g \geq 1$, then: (1) its simple complexity equals $8\pi g$, and (2) its complexity equals $2\pi(m_{\text{min}}+2g-2)$, where $m_{\text{min}}$ is the minimum total length of a branch datum realizable by a branched cover  $p \colon M \to S^2$.
\end{abstract}

\maketitle

\section{Introduction}

In this paper we study the simplest way in which a given Riemann surface can cover the Riemann sphere. The cover is allowed to have branching points and the simplicity of the branched cover is defined in terms of a complexity function. 

Consider a branched cover of a Riemann surface $M$  to the Riemann sphere $S^2$. Let $B \subset S^2$ be the branching set of this cover, and let $X$ be the complement of $B$ in $S^2$. Then $X$ has a natural Riemann surface structure (as a domain of the Riemann sphere $S^2$). We are interested in studying the type of  branched covers described above in which $X$ is a hyperbolic Riemann surface.

The complexity of a branched cover  of a Riemann surface $M$  to the Riemann sphere $S^2$ is defined as the product of the  degree of the cover and the hyperbolic area of the Riemann surface $X$. But, in order for this definition to make sense, we require that the Riemann surface $X$ be hyperbolic, or, equivalently, that the cardinality of $B$ be  greater than or equal to three. [This will be assumed in the paper after Definition \ref{defn-complex}.] 

We first learned of the previous definition of complexity of a branched cover in an unpublished paper by Rieck and Yamashita (\cite{Rieck-Yamashita}), although these authors define  it in the context of branched covers of $3$-manifolds.

We define the \emph{complexity} of a Riemann surface $M$ as the infimum of the complexities of all branched covers of $M$  to $S^2$. The main purpose of this paper is to study the following problem.

\begin{prob} \label{prob-complex}
Find the complexity of a Riemann surface $M$.
\end{prob}

By the Gauss-Bonnet theorem (see Lemma \ref{lem-area-complex}), the complexity of a branched cover of $M$ to $S^2$ of degree $d \geq 1$ and with branching set of cardinality $n \geq 3$ is equal to $2\pi d (n-2)$. Therefore, the set of complexities  of all branched covers of $M$ to  $S^2$ is a subset of $\mathbb{Z} \pi$, the set of all integer multiples of $\pi$, which is a discrete subset of the real numbers.  Thus, the complexity of $M$ is actually equal to the complexity of a branched cover of $M$ to  $S^2$. In other words, Problem \ref{prob-complex} is equivalent to the following.

\begin{prob} \label{prob-complex-min}
Given a Riemann surface $M$, find a branched cover of $M$ to  $S^2$ with minimal complexity.
\end{prob}

Although the complexity of a branched cover is given by a simple formula, Problem \ref{prob-complex-min} is difficult because, given a Riemann surface $M$, there are no known sufficient conditions for the existence of a branched cover of $M$ to $S^2$. Finding such conditions amounts to solving the (still open) Hurwitz problem, which we discuss below.

A \emph{branch datum} is a $4$-tuple $(M,n,d,\Pi)$ such that:
\begin{itemize}
 \item $M$ is a Riemann surface,
 \item $n \geq 0$ and $d \geq 1$ are integers,
 \item $\Pi$ is a collection of $n$ partitions of the integer $d$. 
\end{itemize}

The \emph{total length} of a branch datum $(M,n,d,\Pi)$ is defined as the sum of the lengths of the partitions in the collection $\Pi$.

A branched cover $p \colon M \to S^2$  naturally gives rise to a branch datum $(M,n,d,\Pi)$. Here, $n$ is the cardinality of the branching set $B \subset S^2$, $d$ is the degree of the cover, and $\Pi=(\Pi_1,\Pi_2,\cdots,\Pi_n)$ is the collection such that $\Pi_i$ is the partition of $d$ given by the degrees of the points on the pre-image of the $i^{th}$ branch point in $B$. A branch datum $(M,n,d,\Pi)$ is called \emph{realizable} if it is associated to a branched cover $p \colon M \to S^2$.

Let  $(M,n,d,\Pi)$ be a realizable branch datum, with $\Pi=(\Pi_1,\Pi_2,\cdots,\Pi_n)$. Define $m_i$ as the length of the partition $\Pi_i$ and let $m=\sum_{i=1}^n m_i$. By Theorem \ref{RH-form-alt}, the Riemann-Hurwitz formula becomes: 

\begin{equation} \label{RH-form}
 \chi(M)-m=d(\chi(S^2)-n).
\end{equation}

A branch datum satisfying \eqref{RH-form} is called \emph{compatible}. Thus, every  realizable branch datum is compatible, but the converse is not true (see \cite[Corollary 6.4]{Edmonds}). The classical Hurwitz problem asks  which compatible branch data are actually realizable.

The Hurwitz problem stated above has a natural generalization in which we consider branched covers of an arbitrary surface, not necessarily the sphere $S^2$ (see Section 1 of \cite{Pervova-Petronio-I} for precise statements). This more general problem was studied first by Hurwitz (\cite{Hurwitz}), and more recently by many authors, such as \cite{Baranski}, \cite{Edmonds}, \cite{Gersten}, \cite{Husemoller}, \cite{Pervova-Petronio-I}, \cite{Pervova-Petronio-II} and \cite{Pakovich}. All instances of the general Hurwitz problem either have been solved or can be reduced to the version of the problem stated above, which is the only one that remains open (see Section 2 of \cite{Pervova-Petronio-II}). 

A branched cover $p \colon M \to S^2$ of degree $d$ is \emph{simple} if the cardinality of the pre-image $p^{-1}(y)$ is at least $d-1$ for all $y \in S^2$. We define the \emph{simple complexity} of a Riemann surface $M$ as the infimum of the complexities of all simple branched covers of $M$  to $S^2$. We now state our two main results.

\begin{main-thm} \label{main-thm-intro-1}
 Let $M$ be a connected, closed, orientable Riemann surface of genus $g \geq 1$. Then the simple complexity of $M$ is equal to $8\pi g$.
\end{main-thm}

Main Theorem \ref{main-thm-intro-1} gives an explicit formula for the simple complexity of a Riemann surface in terms of its genus. On the other hand, Main Theorem \ref{main-thm-intro-2} below is not quite explicit: it gives a formula for the complexity of a Riemann surface, but this formula is in terms of an integer  that is difficult to find. More precisely, we have the following.

\begin{main-thm} \label{main-thm-intro-2}
Let $M$ be a connected, closed, orientable Riemann surface of genus $g \geq 1$. Let $m_{\text{min}}$ be the minimum total length of a branch datum realizable by a branched cover  $p \colon M \to S^2$. Then the complexity of $M$ is equal to $2\pi(m_{\text{min}}+2g-2)$.
\end{main-thm}

Given a Riemann surface $M$, it is very difficult to find an explicit formula for the integer $m_{\text{min}}$ from the statement of Main Theorem \ref{main-thm-intro-2}. The reason for this is that we do not know which branch data are realizable (by a branched cover of $M$ to $S^2$), which amounts to the fact that the Hurwitz problem is still open. 

\subsection{Acknowledgements}
The authors would like to thank Daryl Cooper and Ilesanmi Adeboye for  helpful discussions. 

\section{Preliminaries}

\begin{defn} \label{Riemann-surf} 
A \emph{Riemann surface} is a complex manifold of complex dimension one.
\end{defn}

Throughout this paper, $M$ will denote a Riemann surface and $S^2$ will denote the Riemann sphere. All Riemann surfaces in this paper will be connected, closed and orientable. We denote the Euler characteristic of a Riemann surface $M$ by $\chi(M)$.

\begin{defn} \label{defn-branch-cover}
A \emph{branched cover} is a non-constant holomorphic map between Riemann surfaces.
\end{defn}

The following proposition is well-known (see Propositions $5$ and $6$ in Chapter $4$ of \cite{Donaldson}).

\begin{prop} (\cite{Donaldson}) \label{prop-branch-cover}
 Let $p \colon M \to N$ be a branched cover between Riemann surfaces. 
\begin{enumerate}
 \item Then for each $x \in M$ there is a unique integer $k=k_x \geq 1$ such that we can find charts around $x$ in $M$ and $p(x)$ in $N$ in which the map $p$ is represented by the map $z \mapsto z^{k}$.
 \item Let $R$ be the set of all points $x$ in $M$ such that $k_x>1$. Then the set $R$ is finite.
 \item For each $y \in N$, the pre-image $p^{-1}(y)$ is a finite subset of $M$.
\end{enumerate}
\end{prop}

\begin{defn}
With the notation of Proposition \ref{prop-branch-cover}, 
\begin{itemize}
  \item The set $R$ is called the \emph{ramification set} of the branched cover $p$. A point in $R$ is called a \emph{ramification point} of $p$.
 \item The set $B=p(R)$ is called the \emph{branching set} of the branched cover $p$. This set is always finite by Proposition \ref{prop-branch-cover} (2). A point in $B$ is called a \emph{branch point} of $p$.
 \item For each $x \in M$, the integer $k_x \geq 1$ is called the \emph{ramification index} of $x$. 
 \item The \emph{total ramification index} of the branched cover $p$ is defined to be \[\mathcal{R}_p=\displaystyle \sum_{x\in X} (k_x-1).\] [This is really a finite sum by Proposition \ref{prop-branch-cover} (2).]
\end{itemize}
\end{defn}

Given a branched cover $p \colon M \to N$, there is an integer $d \geq 1$ such that every point in $N$ has exactly $d$ pre-images, if we count pre-images with appropriate multiplicities. More precisely, 

\begin{lem} (\cite{Donaldson})
\label{lem-counting mult} Let $p \colon M \to N$ be a branched cover between Riemann surfaces. 
For each $y \in N$, we define the integer  \[d(y)=\displaystyle \sum_{x\in p^{-1}(y)} k_x.\] Then the integer $d(y)$ does not depend on $y$. This integer will be called the \emph{degree} of the branched cover $p$.
\end{lem}

\section{The Riemann-Hurwitz formula and the Gauss-Bonnet theorem}

We now state the classical Riemann-Hurwitz formula.
\begin{thm} \label{RH-form-gen}  (\cite[Theorem 2.5.2]{Jost}) Let $p \colon M \to N$ be a branched cover of degree $d$ between Riemann surfaces. Suppose that the genera of $M$ and $N$ are $g_M$ and $g_N$, respectively. Let $\mathcal{R}_p$ be the total ramification index of the branched cover $p$. Then \[2-2g_M=d(2-2g_N)-\mathcal{R}_p.\]
\end{thm}

We will also use the following alternative way of stating the Riemann-Hurwitz formula for branched covers of the Riemann sphere.

\begin{thm} \label{RH-form-alt}  Let $M$ be a Riemann surface and let $p \colon M \to S^2$ be a branched cover of degree $d$. Suppose that branching set $B \subset S^2$ has cardinality $n$ and that the cardinality of the set $p^{-1}(B) \subset M$ is $m$. Then \[ \chi(M)-m=d(\chi(S^2)-n).\]
\end{thm}

\begin{proof}
 Removing all the $n$ branch points from $S^2$ and all their $m$ pre-images from $M$, we obtain that $p$ restricts to a genuine cover  $p \colon M \setminus p^{-1}(B) \to S^2 \setminus B$ of degree $d$. Therefore, $\chi(M \setminus p^{-1}(B))=d(\chi(S^2 \setminus B))$, i.e., $\chi(M)-m=d(\chi(S^2)-n)$.
\end{proof}

From now on, we will exclusively study the special type of branched covers of the Riemann sphere that we define below.

\begin{defn}
A branched cover $p \colon M \to S^2$ of degree $d$ is \emph{simple} if the cardinality of the pre-image $p^{-1}(y)$ is at least $d-1$ for all $y \in S^2$.
\end{defn}

For simple branched covers of the Riemann sphere, the Riemann-Hurwitz formula simplifies to the following.

\begin{thm} \label{RH-form-sphere} Let $p \colon M \to S^2$ be a simple branched cover  of degree $d$ and with branching set of cardinality $n$. Let $g$ be the genus of $M$. Then  \[2-2g=2d-n.\]
\end{thm}

\begin{proof} By Theorem \ref{RH-form-gen}, it suffices to prove that $\mathcal{R}_p=n$.
Combining the definition of a simple branched cover with Lemma \ref{lem-counting mult}, we obtain:
\begin{itemize}
 \item A point $x \in M$ is a ramification point of $p$ if and only if $k_x=2$.
 \item Given a  branch point $y \in_{n\geq2} S^2$, there exists a unique ramification point $x \in p^{-1}(y)$.
\end{itemize}
Thus, the cardinality of the ramification set $R$ equals the cardinality of the branching set $B$, which is $n$ by assumption. Therefore, \[\mathcal{R}_p=\displaystyle \sum_{x\in X} (k_x-1)=\displaystyle \sum_{x\in R} (k_x-1)=\displaystyle \sum_{x\in R} (2-1)=n.\]
\end{proof}

We recall the classical Gauss-Bonnet Theorem.
\begin{thm} (\cite[Theorem V.2.7]{Chavel}) \label{G-B}
For $M$ orientable with compact closure and smooth boundary, we have 
\[\int_{\partial M} \! \kappa_g \, \mathrm{d} s +\int_{M} \! K \, \mathrm{d} A=2\pi \chi(M)\]
[Here $K$ is the Gaussian curvature of $M$, $\kappa_g$ is the geodesic curvature of the boundary $\partial M$ of $M$, and $\chi(M)$ is the Euler characteristic of $M$.] 
 \end{thm}

\section{$(d,n)$-branched covers of the Riemann sphere}
\begin{defn}
A branched cover $p \colon M \to S^2$ is said to be a \emph{$(d,n)$-branched cover} (of the Riemann sphere) if the following properties are satisfied:
\begin{itemize}
 \item The degree of $p$ is equal to $d$.
 \item The cardinality of the branching set of $p$ is equal to $n$
\end{itemize}
\end{defn}

Let $p \colon M \to S^2$ be a $(d,n)$-branched cover with branching set $B \subset S^2$. Then the surface $S^2 \setminus B$ has a natural structure of a Riemann surface (induced by the Riemann surface structure of $S^2$). This Riemann surface is hyperbolic if and only if $n \geq 3$ (\cite[Theorem 27.12]{Forster}). 

The definition of complexity of a branched cover below (Definition \ref{defn-complex}) require that the complement of the  branching set in the Riemann sphere admit a hyperbolic structure. This will be assumed for the rest of the paper. In other words, 

\begin{notn}
From now on, given a $(d,n)$-branched cover, we will always assume that $n \geq 3$. 
\end{notn}

We now define a complexity on $(d,n)$-branched covers.

\begin{defn} \label{defn-complex}
 The \emph{complexity} of a $(d,n)$-branched cover is defined as $d$ times the hyperbolic area of the complement of its branching set in $S^2$.
\end{defn}

Using the Gauss-Bonnet theorem, we can easily find a formula for the complexity of a $(d,n)$-branched cover.

\begin{lem} \label{lem-area-complex}\hfill
\begin{enumerate}
 \item The hyperbolic area of the complement of $n \geq 3$ points in the sphere equals $2\pi(n-2)$. 
 \item  The complexity of a $(d,n)$-branched cover equals  $2\pi d (n-2)$.
\end{enumerate}
\end{lem}

\begin{proof} \hfill
\begin{enumerate}
 \item Let $M_n$ be the complement of $n \geq 3$ points in the sphere, and let $A$ be its hyperbolic area. By the Gauss-Bonnet theorem (Theorem \ref{G-B}), $A=-2\pi \chi(M_n)=-2\pi (\chi(S^2)-n)=2\pi (n-2)$.
 \item This follows from the definition of complexity.
\end{enumerate}
\end{proof}

We now define the main object of study of this paper.

\begin{defn} \label{defn-simple-complex}
 The  \emph{simple complexity} of a Riemann surface $M$ is the infimum of the complexities of all \emph{simple} branched covers of $M$ to $S^2$.
\end{defn}

\section{The main theorems}

We will need the following non-existence result to prove one of our main theorems.

\begin{lem} \label{lem-no-(1-n)-cov}
There is no $(1,n)$-branched cover with $n \geq 3$.
\end{lem}

\begin{proof}
 Suppose that $p \colon M \to S^2$ is a $(1,n)$-branched cover with $n \geq 3$. Fix a point $y \in S^2$. By Lemma \ref{lem-counting mult}, \[\displaystyle \sum_{x\in p^{-1}(y)} k_x=1.\] Since $k_x \geq 1$ for all $x \in M$, then the last equality implies that the set $p^{-1}(y)$ consists of a unique point $x$ whose ramification index equals $1$. In other words, for each $y \in S^2$, the set $p^{-1}(y)$ contains no ramification points. This means that the branching set of $p$ is empty, which contradicts that $p$ is a $(1,n)$-branched cover with $n \geq 3$.
\end{proof}

\begin{thm} \label{computing-simple-complex}
The simple complexity of a connected, closed, orientable Riemann surface of genus $g \geq 1$ is equal to $8\pi g$.
\end{thm}
\begin{proof}
Let $M$ be a Riemann surface of genus $g \geq 1$ and let $p \colon M \to S^2$ be a simple $(d,n)$-branched cover. By Theorem \ref{RH-form-sphere}, $2-2g=2d-n$, so $n=2(d+g-1)$. Combining this with Lemma \ref{lem-area-complex} $(2)$, we get that  the complexity of $p \colon M \to S^2$ is equal to $2\pi d (n-2)=4\pi d (d+g-2)$.

For a fixed $g\geq 1$, consider the function $f(d)=4\pi d (d+g-2)$, defined for $d \geq 1$. Observe that, by definition, $f(d)$ is the complexity of a simple $(d,n)$-branched cover. The function $f(d)$ is increasing for $d \geq 1$ because $f'(d)=4\pi(d+g-2)+4\pi d=4\pi(2d+g-2)\geq 4\pi$. [The last inequality holds because $d \geq 1$ and $g \geq 1$.]

Let $d_\text{min}$ be the minimal value of $d \geq 1$ such that there exists a simple $(d,n)$-branched cover $M \to S^2$ with $n \geq 3$. The previous paragraph shows  that the simple complexity of $M$ is equal to $f(d_\text{min})$.

By Lemma \ref{lem-no-(1-n)-cov}, then there is no $(1,n)$-branched cover $M \to S^2$. This means that $d_\text{min}>1$. On the other hand, it is well-known that every closed connected  orientable surface is a double branched cover of $S^2$, branched over $n=2(g+1)\geq 3$ points in $S^2$. This means that $d_\text{min}=2$, and so $f(d_\text{min})=f(2)=8\pi g$.
\end{proof}

We would like to have a result similar to Theorem \ref{computing-simple-complex} for computing the complexity of a Riemann surface. We did not succeed in finding such a result, which is a difficult problem that is related to the Hurwitz problem (see the introduction of this paper). However, we have made some progress in this direction.

We now recall some definitions from the Introduction of this paper. 

\begin{defn} \label{simple-complex}
 The  \emph{complexity} of a Riemann surface $M$ is the infimum of the complexities of all branched covers of $M$ to $S^2$.
\end{defn}

A \emph{branch datum} is a $4$-tuple $(M,n,d,\Pi)$ such that:
\begin{itemize}
 \item $M$ is a Riemann surface,
 \item $n \geq 0$ and $d \geq 1$ are integers,
 \item $\Pi$ is a collection of $n$ partitions of the integer $d$. 
\end{itemize}

The \emph{total length} of a branch datum $(M,n,d,\Pi)$ is defined as the sum of the lengths of the partitions in the collection $\Pi$.

A branch datum $(M,n,d,\Pi)$ is \emph{realizable} if there exists a branched cover $p \colon M \to S^2$  such that:
\begin{itemize}
 \item $n$ is the cardinality of the branching set $B \subset S^2$
 \item $d$ is the degree of the cover $p$.
 \item $\Pi=(\Pi_1,\Pi_2,\cdots,\Pi_n)$ is the collection such that $\Pi_i$ is the partition of $d$ given by the degrees of the points on the pre-image of the $i^{th}$ branch point in $B$.
\end{itemize}

\begin{thm} \label{computing-complex}
Let $M$ be a connected, closed, orientable Riemann surface of genus $g \geq 1$. Let $m_{\text{min}}$ be the minimum total length of a branch datum realizable by a branched cover  $p \colon M \to S^2$. Then the complexity of $M$ is equal to $2\pi(m_{\text{min}}+2g-2)$.
\end{thm}

\begin{proof}
Let $p \colon M \to S^2$ be a $(d,n)$-branched cover an let $m$ be the total length of the branch datum associated to $p$. By Lemma \ref{lem-area-complex}, the complexity of $p$ is equal to $2\pi d (n-2)$. On the other hand, by Theorem \ref{RH-form-alt}, we have that \[ \chi(M)-m=d(\chi(S^2)-n).\] Therefore, \[m=\chi(M)-d(2-n)=2-2g+d(n-2),\] and so \[m+2g-2=d(n-2).\] Hence, the complexity of $p$ is equal to $2\pi (m+2g-2)$. Since the genus $g \geq 1$ is fixed, then the minimal complexity of a branched cover $p \colon M \to S^2$ is equal to  $2\pi(m_{\text{min}}+2g-2)$. The conclusion now follows from the definition of complexity of a Riemann surface.
\end{proof}

 \bibliographystyle{alpha}
\bibliography{biblio}

\end{document}